\documentclass{article}
\usepackage[dvips]{graphicx}     
\usepackage[english]{babel}     
\usepackage[latin1]{inputenc}   
\usepackage{makeidx}
\usepackage{graphicx} 
\usepackage{amssymb}
\usepackage{amsmath}
\usepackage{indentfirst}     
\usepackage{xcolor,graphicx}
\usepackage{rotating}
\usepackage{tikz}
\usepgflibrary{shapes.misc}
\usepackage{eurosym}
\usetikzlibrary{matrix}
\usetikzlibrary{%
   arrows,%
   calc,%
   fit,%
   patterns,%
   plotmarks,%
   shapes.geometric,%
   shapes.misc,%
   shapes.symbols,%
   shapes.arrows,%
   shapes.callouts,%
   shapes.multipart,%
   shapes.gates.logic.US,%
   shapes.gates.logic.IEC,%
   er,%
   automata,%
   backgrounds,%
   chains,%
   topaths,%
   trees,%
   petri,%
   mindmap,%
   matrix,%
   calendar,%
   folding,%
   fadings,%
   through,%
   positioning,%
   scopes,%
   decorations.fractals,%
   decorations.shapes,%
   decorations.text,%
   decorations.pathmorphing,%
   decorations.pathreplacing,%
   decorations.footprints,%
   decorations.markings,%
   shadows}

\makeindex
\newtheorem{remark}{Remark}

\newtheorem{lemma}[remark]{Lemma}
\newtheorem{theorem}[remark]{Theorem}

\newtheorem{proposition}[remark]{Proposition}

\newtheorem{corollary}[remark]{Corollary}
\newenvironment{proof}{\begin{trivlist} \item[] {\em Proof:}}{\hfill
$\otimes$\end{trivlist}}
\newtheorem{JSVdef}[remark]{Definition}
\newenvironment{definition}{\begin{JSVdef}\em}{\end{JSVdef}}
\newtheorem{JSVnote}[remark]{Note}
\newenvironment{note}{\begin{JSVnote}\em}{\end{JSVnote}}
\newtheorem{JSVex}[remark]{Example}
{\unskip\nobreak\hskip 1em plus 1fil\nobreak$\Box$
\parfillskip=0pt
\end{JSVex}}

\newcommand{\R}{\mbox{\rm {I$\!$R}}}

\newcommand{\D}{\displaystyle}

\title{
  \Large {
  \bf    Locating Two Transfer Points on a Network with a Trip Covering Criterion and Mixed Distances
 }
   }
     \author{ 
M.C. L\'opez-de-los-Mozos \thanks {
Dpto.  Matem\'atica Aplicada I.  Universidad de Sevilla,
 Spain.  mclopez@us.es}
\and Juan A. Mesa   \thanks {Dpto.  Matem\'atica Aplicada II.
Universidad de Sevilla, Spain.
jmesa@us.es}
\and Anita Sch\"{o}bel \thanks {Institut f\"{u}r Numerische and Angewandte Mathematik.
Universit\"{a}t G\"{o}ttingen, Germany.   schoebel@math.uni-goettingen.de}
}
\date{}
\begin{document}
\maketitle
\begin{abstract}

In this paper we consider a set of origin-destination pairs  in a mixed model in which a network embedded in the plane represents an alternative  high-speed transportation system, and study a trip covering problem which consists on 
locating two points in the network which maximize the number of covered pairs, that is, the number of pairs which use the network by acceding and exiting through such points.
To deal with the absence of convexity of this mixed distance function we propose a  decomposition method based on formulating a collection of  
subproblems
and solving each of them via discretization of the solution set.

\end{abstract}

\noindent
{\bf Keywords}: Location; Network; Covering Problem; Mixed Distances

\bigskip  
\section{Introduction}
In a  previous paper~\cite{KorMesa11}, a location problem on a mixed planar-network space was tackled. In the problem dealt with in that paper, there are existing facilities in the Euclidean plane whereas new facilities are to be located in a straight-line segment or in a tree network. Instead of covering facilities the objective aims at covering trips between each pair of existing facilities. These trips can be done either by using the plane with the Euclidean distance or by a combination plane-network in which the section on the network is supposed to be traversed faster than those on the plane. Therefore, there is a competition between the mode that only uses the planar distance and the combined one. Each pair of existing facilities has an associated demand and the objective is to maximize the number of trips for which the combined mode is preferable to the planar one. In this paper we extend the approach applied in~\cite{KorMesa11}  
to the case of general networks
by taking into account the loss of convexity of the distance function between pairs of points through the network.

The problem of locating stations in a railway network was indirectly tackled in~\cite{VuNe}, in which they considered the problem of optimal   interstation spacing in a commuter line. The objective of this problem was to minimize the total time of passengers going to a city center along a railway line. A commuter line competing with a freeway was considered in~\cite{Vu} in order to maximize the number of passengers on the basis if the shortest travel time. The aim was also to determine the optimal  interstation space between pairs of adjacent stations. Apart from the 
papers~\cite{LaMeOr, LaMeOrSe} in which the feasible solution space for locating stations was discrete, and the objective was to maximize the passengers and trip coverage, respectively,  several objective functions have been considered in those in which stations can be located along the edges of the railway network (\cite{Haetal, Scho, Schoetal, ReAnChurch}).
 
In order to avoid duplications  the reader is referred to the Introduction of the paper~\cite{KorMesa11} for an overview of other related papers.
 
The case of locating only one facility regarding transfer facilities already located in the network is a particular case of that dealt with in this paper and can be categorized as a conditional location problem. Thus, adding  the extra demand captured by both pairs consisting of a new facility and an already located  one, and those consisting of two new facilities, the global contribution of the two new transfer points is computed.
 
The paper is organized as follows. After this introduction the problem is formulated in Section 2. The necessary definitions and results on the distance on networks are summarized in Section 3. Section 4 is devoted to solve the two cases arisen from the properties of the distance. The paper ends with some conclusions and further research.

\section{The problem} 
Hereinafter we will use the notation introduced by~\cite{KorMesa11}.
Let ${\cal A} = \{ A_i=(a_{i}, b_i), i=1, \ldots, n \} \subset {\R}^2$ be a set of existing facilities on the plane. We assume that   travel time distances between two points in the plane can be estimated by the Euclidean metric.

Let $T= (t_{ij}) \in {\R}^{n \times n}$ be an origin-destination (O/D) matrix in which trip patterns are codified, i.e., $t_{ij}$ is the weight of the ordered  pair $(A_i, A_j)$ (or $(i, j)$, if there is no confusion). This matrix is known a priori:  for example, in a transportation context each $t_{ij}$ can be viewed as the number of trips from an origin $A_i$ to a destination $A_j$, and in a telecommunication setting  it could represent the amount of data transfered from server $i$ to server $j$.

In order to formulate the problem with a mixed mode of transportation, we consider an embedded  network ${\cal N} (V, E)$ representing a high-speed system, with $|V|$ vertices and $|E|$ edges. 
Each vertex $v \in V$ represents a junction or a node,
and we assume that each undirected edge $e \in E$  
 has a length $l_e$ and it can be modeled as a straight-line segment.
The embedding of ${\cal N}$  in the Euclidean plane  as well as the coordinates of the nodes in ${\cal N}$ will allow us to compute the travel distances between each pair of points.  Let ${\cal N}$ be the continuum set of points on the edges. The edge lengths induce a distance function $d$ on ${\cal N}$, such that for any  two points $x, y \in {\cal N}$, $d(x,y)$ is the length of any shortest path connecting $x$ and $y$. Moreover, if $x$ and $y$ are on the same edge, $d(x,y)= || x-y ||$.

For any two points $x, y \in {\cal N}$, let us define the high-speed distance $d_{\cal N} (x, y) = \alpha d(x, y)$, with $\alpha \in (0, 1)$. Parameter $\alpha$ is a speed factor. It is straightforward to see that $({\cal N}, d_{\cal N})$ is a metric space.

Given an O/D pair $(i, j)$, the transportation time by using the network is obtained by selecting  two  points  $X_1 \in {\cal N}$ and  $X_2 \in {\cal N}$ (or a 2-facility point $(X_1, X_2) \in {\cal N}^2$) such that the transportation path from $i$ to $j$ enters the network at one of such points and  exits the network at the other one. If  $X_1$ is the access point from $A_i$ and $X_2$ the exit point towards $A_j$,  the length of the travel-path $(A_i, X_1, X_2, A_j)$ is given by:
$$h_{ij}(X_1, X_2)  = ||A_i-X_1||_2 + d_{\cal N} (X_1, X_2)+ ||X_2-A_j||_2$$
Moreover, if $X_2$ is the access point from $A_i$ and $X_1$ the exit point towards $A_j$, the length of the travel-path    
$(A_i, X_2, X_1, A_j)$ is  
$$h_{ij} (X_2, X_1) = ||A_i-X_2||_2 + d_{\cal N} (X_2, X_1)+ ||X_1-A_j||_2 $$
Although  algebraically  $h_{ij}(X_1, X_2)=h_{ji}(X_2, X_1)$ and $h_{ij}(X_2, X_1)=h_{ji}(X_1, X_2)$, in this context the subindex $ij$ indicates the origin/destination pair $(i, j)$, which is different from  the O/D pair $(j, i)$. For this reason, with each O/D pair $(i, j)$ we only associate the  subindex $ij$. 

The distance between the O/D pair $(A_i, A_j)$ by using the transit points $X_1, X_2 \in {\cal N}$ is:
$$ f_{ij} (X_1, X_2) = \min \{ h_{ij}(X_1, X_2), h_{ij}(X_2, X_1)\} = f_{ij} (X_2, X_1)$$
Let $D=(d_{ij}) \in {\R}^{n \times n}$ be a symmetric matrix, with 
$0 \le d_{ij} < || A_i-A_j||_2$.  The values of $D$ represent the {\it acceptance levels} for using the network, meaning that the O/D pair $(i, j)$ chooses the high-speed network if and only if the traveling time by using it is  less than or equal to
$d_{ij}$. In  other words, the O/D pairs always choose the faster option.

\begin{definition} \label{cover}
The O/D pair $(i, j)$ is covered by  $(X_1, X_2) \in {\cal N}^2$ if and only if 
$f_{ij}(X_1, X_2) \le d_{ij}$.
\end{definition}
 Let $C(X_1, X_2)$ be the set of O/D pairs covered by $(X_1, X_2)$, that is:
$$ C(X_1, X_2) = \{ (A_i, A_j) \in {\cal A} \times {\cal A} : \, f_{ij}(X_1, X_2) \le d_{ij} \}$$
Symmetry of the acceptance level matrix  and both the Euclidean and the network distances implies that $(i, j) \in C(X_1, X_2)$ if and only if $(j, i) \in C(X_1, X_2)$.

The objective function measures the amount of trip patterns
captured by 
each 2-facility-point $(X_1, X_2) \in {\cal N}^2$, and it 
 is given by the 
trip-sum function:
$$ F(X_1, X_2) = \sum_{(i, j) \in C(X_1, X_2)} t_{ij}$$ 
Finally, the O/D-pair 2-location problem is  to find a point $(X_1, X_2) \in {\cal N}^2$ such that the sum of trip patterns of all O/D pairs covered by such a point is maximized:
$$   \begin{array}{ll} \max & F(X_1, X_2) := \D \sum_{(i, j) \in C(X_1, X_2)} t_{ij} \\[4ex]
\mbox{s.t.} &  (X_1, X_2) \in {\cal N}^2
\end{array} \eqno(1)$$

The problem (1) has been first solved  for the particular case where ${\cal N}$  is a segment of a straight line,  and later the method was  extended to  the case where 
${\cal N}$ is a tree network ${\cal T}$ (see ~\cite{KorMesa11}). However, the approach applied for the tree network case cannot be directly extended to a general network ${\cal N}$ due to the absence of convexity of distance functions on a cyclic network.  Ir order to solve the problem, we next summarize some concepts on the behavior of distance on  networks.
\section{Previous results on distances on networks}
Let $e = [u, w] \in E$ be an edge of the  network ${\cal N}$ of length $l_e=l(u, w)$, and let $P$ be a point on $e$. 
 Let $x=l(u, P)$ be the length of the subedge from the left vertex $u$ to $P$, and $l(w, P) = l_e-x$ be the length of subedge  $[P, w]$.   It is well known that for any node $v_i \in V$, the distance $d(v_i, P)= d_i(x)=\min \{ d(v_i, u)+x, d(v_i, w) + l_e-x \} $ is   
concave and piecewise linear on $x \in [0, l_e]$,  with at most two pieces with slopes $1$ and $-1$. If $l(u, w) \le | d(v_i, u)-d(v_i, w)|$, then the distance function $d(v_i, P)$ on $[u, w]$ is linear. The following concepts and definitions can be found in~\cite{HoGarChen91} and references therein.

\begin{definition} Given a point $Q \in  {\cal N}$ and an edge $[u, w]$, the point $P \in [u, w]$ at which $d(Q, P)$ is maximized is called {\it antipodal} to $Q$ in $[u, w]$.
\end{definition}

Note that if $P$ is antipodal to $Q$ does not imply that $Q$ is antipodal to $P$; if it is $P$ and $Q$ are antipodal to each other. 

\begin{definition}
A point $\bar{v}_i \in [u, w]$, other than a node,   is called an {\it arc bottleneck point} if there is a vertex $v_i$ for which  $\bar{v}_i $ is antipodal to $v_i$. In this case,  $d(v_i, u)+l(u, \bar{v}_i) = d(v_i, w)+l(w, \bar{v}_i)$. 
\end{definition}

 Clearly, $[u, w]$ contains at most $|V|$ arc bottleneck points (since each vertex $v_i$ defines at most one arc bottleneck point  $\bar{v}_i$ on the edge).  Such a point  $\bar{v}_i$ can be identified by the length of the subarc $[u, \bar{v}_i]$ given by: 
$ l(u, \bar{v}_i) = \D \frac{d(v_i, w)-d(v_i, u)+l(u, w)}{2}$.

\begin{definition} Let $B_e$ be the set of arc bottleneck points of edge $e=[u, w]$, and
let $\bar{v}_i$, $\bar{v}_j$ be two adjacent points of $B_e \cup \{u, w\}$. Then  $L_i=[ \bar{v}_i, \bar{v}_j]$, the closed subedge limited by such points, is called a {\it linear arc segment} (or a treelike 
segment~\cite{HoGarChen91}).
\end{definition}
 On each linear arc segment the distance $d(v_i, P)$ is linear. If $B_e = \emptyset$ the entire edge $[u, w]$ (including  nodes) is a linear arc segment. The set of all linear arc segments of an edge $e$ is denoted by ${\cal L}(e)$.

Let $[u_p, w_p]$ and $[u_q, w_q]$ be two edges of the network $\cal N$, and let 
$[\bar{w}_p, \bar{u}_p] \subseteq   [u_q, w_q]$, $[\bar{w}_q, \bar{u}_q] \subseteq   [u_p, w_p]$ be the subedges for which the extreme points are antipodal points as follows:     $\bar{u}_p, \bar{w}_p $ are the antipodal points to $u_p, w_p$ respectively,  and
$\bar{u}_q, \bar{w}_q$ are  the antipodal points  to $u_q, w_q$, respectively (see Figure 1(a)).  From previous definitions we have
$l(\bar{w}_q, \bar{u}_q)=l(\bar{w}_p, \bar{u}_p)$ and
$d(\bar{w}_q, \bar{w}_p)=d(\bar{u}_q, \bar{u}_p)$ (see ~\cite{HoGarChen91} for a more detailed explanation).

\begin{center}
\begin{tabular}{ccc}
\begin{tikzpicture}[scale=0.55]
\draw [gray] (0,0) -- (7,0);
\draw [gray] (0,0) -- (1,4);
\draw [gray] (1,4) -- (6,4);
\draw [gray] (6,4) -- (7,0);
\draw [white] (4,5.5) circle (0pt);
\fill [black] (0,0) circle (3pt) node[black, below] {$u_q$} ;
\draw [black] (1, 4) node[above] {$u_p$};
\fill [black] (1,4) circle (3pt) ;
\fill [black] (6,4) circle (3pt)  node[black, above] {$w_p$} ;
\fill [black] (7,0) circle (3pt)  node[black, below] {$w_q$};
\draw [gray] (0,0)--(1, 2/3);
\draw [gray, dashed] (1, 2/3)--(5, 10/3);
\draw [gray] (5, 10/3)--(6,4); 
\draw [gray] (1,4)--(2, 10/3);
\draw [gray, dashed] (2, 10/3)--(5, 4/3);
\draw [gray] (5, 4/3)--(7,0);
\fill [red] (4.5,4) circle (3pt)node[black,above] {$\overline{u}_q$};
\fill [red] (2.5,4) circle (3pt)node[black,above] {$\overline{w}_q$};
\draw [red] (2.5, 4)--(4.5,4); 
\fill [red] (5.2,0) circle (3pt) node[black, below] {$\overline{u}_p$};
\fill [red] (1.8,0) circle (3pt) node[black, below] {$\overline{w}_p$};
\draw [red] (1.8, 0)--(5.2,0); 
\end{tikzpicture} 
&
&
\begin{tikzpicture}[scale=0.5]
\draw [gray] (0,0) -- (7,0);
\fill [black] (0,0) circle (3pt)  node[black, above] {$u$} ;
\fill [black] (7, 0) circle (3pt)  node[black, above] {$w$} ;
\fill [red] (2,0) circle (3pt) node[black, above] {$\overline{w}$};
\fill [red] (5,0) circle (3pt) node[black,above] {$\overline{u}$};
\draw [red]  (2, 0)--(5, 0);
\draw [dashed, domain=0:7] plot(\x, {4/49*\x^2-4/7*\x});
\draw [black] (3.5, -1) node[below] {$d(u, w)<l(u, w)$};
\end{tikzpicture}  \\[1ex]
(a) & & (b) 
\end{tabular}
\end{center}

{\centerline{Figure 1. Antipodal points on: (a) different edges, (b) the same edge}}

\begin{definition}
Let $L_p$, $L_q$ be two linear arc segments of different edges $[u_p, w_p]$ and $[u_q, w_q]$, respectively, and let $[\bar{w}_q, \bar{u}_q] \subseteq   [u_p, w_p]$ and $[\bar{w}_p, \bar{u}_p] \subseteq   [u_q, w_q]$ be the    
      subedges for which
$\bar{u}_p, \bar{w}_p $ are the antipodal points to $u_p, w_p$ respectively,  and
$\bar{u}_q, \bar{w}_q$ are  the antipodal points  to $u_q, w_q$, respectively.
 If $L_p \subseteq [\bar{w}_q, \bar{u}_q]$ and  $L_q \subseteq [\bar{w}_p, \bar{u}_p]$, then $L_p$, $L_q$ are  called antipodal  segments to each other.
\end{definition}
(This definition includes the case $[\bar{w}_q, \bar{u}_q] = [u_p, w_p]$ and $[\bar{w}_p, \bar{u}_p] =   [u_q, w_q]$).
The following result  summarizes   
the behavior of distance $d(P, Q)$ on linear arc segments of different edges (see Hooker, Garfinkel and Chen~\cite{HoGarChen91}).

\begin{lemma} \label{lema6}
Let $P$ be restricted to linear arc segment $L_p$ of edge $[u_p, w_p]$ and $Q$ to linear arc segment $L_q$  of edge $[u_q, w_q]$, with  $[u_p, w_p] \neq [u_q, w_q] $. 
\begin{enumerate}
\item \  If $L_p, L_q$ are antipodal  segments to each other,  $d(P, Q)$ is concave. 
\item \  Otherwise, $d(P, Q)$ is linear.
\end{enumerate}
\end{lemma}

In case 1, the distance between a point $P$ on the segment $[\bar{w}_q, \bar{u}_q]$ and a point $Q$ on the segment $[\bar{w}_p, \bar{u}_p]$ behaves like the distance on the parallelogram in Figure 1(a), and it is concave. Distance on any other pair of segments behaves like distance on a line segment and is therefore linear.

By denoting $x=l(u_p, P)$, with $x \in  [0, l(u_p, w_p)]$ and
$y=l(u_q, Q)$, with $y \in  [0, l(u_q, w_q)]$, we can compute the  distance $d(P, Q)$ for the cases considered in this lemma. 
$$d(P, Q) = \left\{ \begin{array}{l}
\min\{ x + d(u_p, u_q) + y, \; l(u_p, w_p)-x + d(w_p, w_q) + l(u_q, w_q)-y \}, \; \;   L_p, L_q \;  \mbox{antipodal} \\[2ex]
x + d(u_p, u_q) + y,  \quad L_p \subseteq [u_p, \bar{w}_q],   L_q \subseteq [u_q, \bar{w}_p] \;
\mbox{or}  \;  L_q \subseteq [\bar{w}_p, \bar{u}_p] \\[2ex]
x + d(u_p, w_q) + l(u_q, w_q)-y, \quad   L_p \subseteq [u_p, \bar{w}_q],  L_q \subseteq [\bar{u}_p, w_q] 
\end{array} \right. $$
Note that in the case $L_p, L_q$ antipodal segments, with 
$u_p \neq \bar{w}_q$, $\bar{u}_q \neq w_p$ 
and similarly 
$u_q \neq \bar{w}_p$, $\bar{u}_p \neq w_q$, the distance $d(P, Q)$ can also be computed by $\min \{ x + d(u_p, w_q)+l(u_q, w_q)-y, \; l(u_p, w_p)-x+d(w_p, u_q)+y \}$. 
Finally, the remaining cases obtained by combining $L_p, L_q$  can be reduced  to one of these by symmetry  (for example, for the case $L_p \subseteq [u_p, \bar{w}_q],  L_q \subseteq [\bar{u}_p, w_q]$ the distance can also be equivalently obtained as $l(u_p, w_p) -x+ d(w_p, u_q)+y$). 

We now study the case in which $P, Q$ lie on the same edge $ [u, w]$, with $d(u, w) \le l(u, w)$. 
Let  $\bar{u}$, $\bar{w}$ be the antipodal points to $u, w$ respectively. 
The sequence $\{ u, \bar{w}, \bar{u}, w\}$ induces a partition of $[u, w]$ in (at most) three consecutive subedges: $[u, \bar{w}]$, $[\bar{w}, \bar{u}]$ and $[\bar{u}, w]$ (Figure 1(b)).       
 \begin{lemma} \label{lema7}
Let $L_p$, $L_q$ be two linear arc segments  of $[u, w]$ such that $P $ is restricted to $L_p$ and $Q$ is restricted to $L_q$. 
\begin{enumerate}
\item \ If $L_p=L_q$ the distance $d(P, Q)$ is convex. 
 \item \ If $L_p \neq L_q$, and $L_p$, $L_q$ lie respectively in non-adjacent subedges of the partition, $d(P, Q)$  is  concave.
\item \ Otherwise, $d(P, Q)$ is linear. 
\end{enumerate}
\end{lemma}
Clearly, if $L_p=L_q$ the distance $d(P, Q)$ is convex because the arc segment between $P$ and $Q$ is the shortest path between them, and it is concave if 
$L_p \subseteq [u, \bar{w}]$, $L_q \subseteq [ \bar{u}, w]$ (or vice-versa), with $\bar{w} \neq \bar{u}$ (which arises when $d(u, w)<l(u, w)$, that is, when the edge $[u, w]$ is not the shortest path between $u, w$, as in Figure 1(b)).   As above, by denoting $x = l(u, P)$ and $y=l(u, Q)$,  we have: 
 $$d(P, Q)= \left\{ \begin{array}{l}
|x- y|, \quad  L_p = L_q \\ [2ex]
\min \{ y-x, \; x+d(u, w) + l(u, w)-y \},  \; \;   L_p \subseteq [u, \bar{w}], L_q \subseteq [ \bar{u}, w],  \bar{w} \neq \bar{u} \\[2ex]
|y-x|, \; \mbox{otherwise} 
\end{array} \right. $$
Note that in the third case, the distance $|y-x|$ becomes linear, according with the order in which $L_p$ and $L_q$ are placed. 
If $d(u, w)=l(u, w)$, then $u=\bar{w}$ and $\bar{u}=w$, therefore the distance $d(P, Q)$ is convex (if $L_p=L_q$), or linear (if $L_p \neq L_q$).

\section{Solving the problem on a network}
The solution method is based on decomposing  the problem (1) in a collection of subproblems independent to each other, and solving each of them via discretization of the solution set. 
To this end and with purposes  of readability, we will describe the process in three phases: the first one is devoted to both decomposing the problem (1) and  identifying the two type  of subproblems, the second one deals with the subproblems of type 1, and finally in the last step we will study the subproblems of type 2. Next it will be described both type of subproblems.

\subsection{Decomposing the problem}
To decompose the problem we first need the matrix  $\Delta= \D \left( d(v_i, v_j) \right)_{i,j}$ of shortest distances between all pairs of nodes of the network ${\cal N}$. This matrix can be calculated in $O(|V| |E|)$ running time. 
 Then, for each edge $e \in E$ we compute, and sort,  the arc bottleneck points of the set $B_e$ (in $O (|V| \log |V|)$ time).	At the end of this process we have, for each edge $e$ of the network, the ordered sequence ${\cal L}(e)$ of all linear arc segments of the edge. Besides, given a pair of different edges $e_p=[u_p, w_p]$ and $e_q=[u_q, w_q]$ we  know if two linear arc segments
$L_p \in {\cal L}(e_p)$ and  $L_q \in  {\cal L}(e_q)$ are, or not, antipodal to each other. If $e_p$ and $e_q$ are the same edge
  $e=[u, w]$ (which is the case of Lemma~\ref{lema7}), we also know the location of the linear arc segments on the subedges of the partition induced by $\{u, \bar{w}, \bar{u}, w\}$.

Let ${\cal L} =
\D \bigcup_{e \in E} {\cal L}(e) $  be the set of all linear arc segments of the overall network. Thus  problem (1)   is decomposed into a set of subproblems, where each of them is obtained by restricting the feasible space  ${\cal N}\times {\cal N}$ to $L_p \times L_q$, with  $L_p, L_q \in {\cal L}$. By imposing that $X_1 \in L_p$, $X_2 \in L_q$  we obtain the restricted problem: 
$$  \begin{array}{ll} \max & F(X_1, X_2) := \D \sum_{(i, j) \in C(X_1, X_2)} t_{ij} \\[4ex]
\mbox{s.t.} &  (X_1, X_2) \in L_p \times L_q, \;   L_p, L_q \in {\cal L}
\end{array}  \eqno(2)$$
A solution to problem (1) is found by solving the collection of all $|{\cal L}|^2$ restricted problems (2), and then selecting the best solution.

The classification of the restricted problem (2) is done according to the concavity of the distance
$d(X_1, X_2)$.  If $d(X_1, X_2)$ is either linear or convex, we classify the restricted problem  as type 2. Otherwise, as type 1.
 More specifically,
consider the restricted problem formulated in (2). Then the pair $\{ L_p, L_q \}$ is called of type 1 if one  of the following two conditions holds:
\begin{description}
\item{(Lemma 6-1)}: $L_p \in {\cal L}(e_p)$, $L_q \in {\cal L}(e_q)$, with $e_p\neq e_q$, and $L_p$, $L_q$ are antipodal to each other.
\item{(Lemma 7-2)}: $L_p, L_q \in {\cal L}(e)$, with $e=[u, w]$ such that $d(u, w)<l(e)$, and 
$L_p, L_q$ lie in not adjacent subedges of partition $\{u, \bar{w}, \bar{u}, w\}$.
\end{description}
For the remaining cases, the pair $\{L_p, L_q\}$ is said of type 2.

\subsection{The concave case: The restricted problem of type 1}
The problem can be formulated as follows
$$    \begin{array}{ll} \max & F(X_1, X_2) := \D \sum_{(i, j) \in C(X_1, X_2)} t_{ij} \\[4ex]
\mbox{s.t.} &  (X_1, X_2) \in L_p \times L_q   \\[1ex]
\mbox{} &  L_p, L_q \in {\cal L}, \; \{L_p, L_q\} \; \mbox{of type 1}
\end{array}  \eqno(3)$$

Each linear arc segment is a rectifiable subedge in which all distance functions $d(v_i, P)$ are linear, and each vertex $v_i \in V$ links with all points of such subedge via one of their extreme points. On the other hand, we have:
\begin{enumerate}
\item   For each $A_i \in {\cal A}$, the Euclidean distance $||A_i-X||_2$   is convex when $X $  varies in any linear arc segment. 
In effect, from convexity theory  (see~\cite{Boyd, Rockafeller}), 
if $f$ is a convex function in an open set $U$ then the restriction of $f$ to any interval (i.e., straight line segment) inside $U$ is convex. Since every norm on ${\R}^n$ is convex, the result follows straightforwardly.

\item  When $(X_1, X_2) \in L_p \times L_q$, with $\{ L_p, L_q\}$ of type 1, the distance
$d(X_1, X_2)$ is concave (Lemmas~\ref{lema6} and~\ref{lema7}), and also is concave $d_{\cal N}(X_1, X_2) = \alpha \, d(X_1, X_2)$ (where $\alpha \in (0, 1)$ is the speed factor). 
\end{enumerate}

Therefore the function
$h_{ij}(X_1, X_2)=||A_i-X_1||_2+d_{\cal N}(X_1, X_2) + ||X_2-A_j||_2$
is neither convex nor quasiconvex on $L_p \times L_q$ (and similarly for $h_{ij}(X_2, X_1)$).

To illustrate some concepts and properties associated to this case, we will use the  small network with trapezoidal shape shown in Figure 2. 
Length of  basis of trapezoid are 7 and 5, respectively,  and the four vertices 
$u_p, w_p, u_q$ and $w_q$ of network are the points $(0, 2\sqrt{6})$,  $(5,  2\sqrt{6})$, $(-1, 0)$ and $(6, 0)$, respectively. The only arc bottleneck points are
$\bar{w}_p=(0, 0)$ and $\bar{u}_p=(5, 0)$  opposite to $w_p$ and $u_p$, respectively. By selecting $L_p=[u_p, w_p]$ and
$L_q=[\bar{w}_p, \bar{u}_p]$, then $L_p$, $L_q$ are antipodal to each other 
 (note that in this example, $L_p$ is the whole edge).
Subsequent examples on this network deal with the O/D pair $(i, j)$, corresponding to points $A_i(2.5, 6)$ and $A_j(1, -4)$.

\bigskip

\begin{center}
\begin{tikzpicture}[scale=0.55]
\draw  [black] (-4, -4.8)--(10, -4.8)--(10, 7.5)--(-4, 7.5)--cycle;
\draw [->, dashed, gray] (0,0)--(0,7); 
\draw [->, dashed, gray] (0,0)--(8,0); 
\draw [gray] (-1,0) -- (6,0);
\draw [gray] (-1,0) -- (0,4.89897);
\draw [gray] (0,4.89897) -- (5,4.89897);
\draw [gray] (6,0) -- (5,4.89897);
\fill [black] (-1,0) circle (3pt);
\draw (-1.8, 0) node[black, below] {{\scriptsize{$(-1,0)$}}} ;
\draw [black] (-1.5, 0) node[above] {$u_q$};
\fill [black] (6,0) circle (3pt);
\draw (6.5,0)  node[black, below] {\scriptsize{$(6,0)$}};
\draw [black] (6.5, 0) node[above] {$w_q$};
\fill [black]  (0,4.89897) circle (3pt);
\draw (-1.2, 4.89897) node[black, above] {\scriptsize{$(0, 2 \sqrt{6})$}} ;
\draw [black] (0,4.7) node[left] {$u_p$};
\fill [black]  (5,4.89897) circle (3pt);
\draw  (6, 4.89897) node[black, above] {\scriptsize{$(5, 2 \sqrt{6})$}} ;
\draw (-0.5,2.5)  node[black, left] {\scriptsize{$5$}};
\draw (5.5,2.5)  node[black, right] {\scriptsize{$5$}};
\draw [black] (5,4.7) node[right] {$w_p$};
\draw [blue, thick]  (0,4.89897) -- (5,4.89897);
\fill [blue]  (0,4.89897) circle (3pt);
\fill [blue]  (5,4.89897) circle (3pt);
\draw [blue, thick]  (0,0) -- (5,0);
\draw (0.2,0) node[black, below] {\scriptsize{$(0,0)$}};
\draw (4.8,0) node[black, below] {\scriptsize{$(5,0)$}};
\fill [blue]  (0,0) circle (3pt);
\fill [blue]  (5,0) circle (3pt);
\draw [blue] (2.5,4.899) node[below] {$L_p$};
\draw [blue] (2.5,0) node[above] {$L_q$};
\draw (0.6,0) node[black, above] {$\overline{w}_p$};
\draw (4.8,0) node[black, above] {$\overline{u}_p$};
\fill [black]  (2.5, 6) circle (3pt) node[above] {$A_i (2{.}5, 6)$};
\fill [black]  (1, -4) circle (3pt) node[above] {$A_j (1, -4)$};
\end{tikzpicture}
\end{center}

{\centerline{Figure 2. Antipodal linear arc segments $L_p=[\bar{w}_q, \bar{u}_q]=[u_p, w_p]$ and
$L_q=[\bar{w}_p, \bar{u}_p]$}}

\bigskip

The strategy for solving this problem is based on identifying a Finite Dominating Set (FDS) such that an optimal solution can be found by selecting the best solution from the FDS.
To this end we first introduce  some definitions and a collection of sublevel sets which will be used in the construction of such a set.

When $(X_1, X_2) \in L_p \times L_q$,   and $\{ L_p, L_q\}$ are two linear arc segments of type 1,  the distance $d(X_1, X_2)$ is a concave function obtained from the lower envelope of two linear functions: $d_a(X_1, X_2)$ and $d_b(X_1, X_2)$. That is,  $d(X_1, X_2) = \min \{ d_a(X_1, X_2), d_b(X_1, X_2) \}$, where this expression is obtained from
Lemmas~\ref{lema6} and~\ref{lema7} (according to whether or not  $L_p$ and $L_q$ lie in different edges), by replacing   $P $ and $Q $  by $X_1$ and $X_2$, respectively. More specifically,
\begin{enumerate}
\item From Lemma~\ref{lema6}, 
$$ d_a(X_1, X_2)=x + d(u_p, u_q) + y, \; \mbox{and} \;   d_b(X_1, X_2)= l(u_p, w_p)-x + d(w_p, w_q) + l(u_q, w_q)-y $$
\item  From Lemma~\ref{lema7}, 
$$d_a(X_1, X_2)=y-x, \; \mbox{and} \;  d_b(X_1, X_2)= x+d(u, v)+l(u, w)-y$$
\end{enumerate} 

For simplicity of notation, henceforth we will use $(x, y)$ and $(X_1, X_2)$ indistinctly.
In the example of Figure 2,  denoting by $x$ and $y$ the lengths of subedges 
$[u_p, X_1]$ and $[\bar{w}_p, X_2]$, respectively, the distance $d(X_1, X_2)$ is given by
$$ d(X_1, X_2) = \min \{ 6+x+y, \, 16-x-y \}, \; \mbox{with} \; (x, y) \in [0, 5] \times [0, 5] $$

 Therefore, the length $h_{ij}(X_1, X_2)$ of travel path $(A_i, X_1, X_2, A_j)$ can be expressed as:
$$ h_{ij} (X_1, X_2)=||A_i-X_1||_2 + \alpha \, \min \{ d_a(X_1, X_2), d_b(X_1, X_2) \} + ||X_2-A_j||_2 $$
and consequently 
$$ h_{ij} (X_1, X_2)= \min\{ g_{ij}^a (X_1, X_2), g_{ij}^b (X_1, X_2) \}$$
where 
$$ \begin{array}{l} g_{ij}^a (X_1, X_2)=||A_i-X_1||_2+\alpha \, d_a(X_1, X_2)+ ||X_2-A_j||_2 \\[1ex] 
g_{ij}^b (X_1, X_2)= ||A_i-X_1||_2+\alpha \, d_b(X_1, X_2)+ ||X_2-A_j||_2
\end{array} $$
  Since $d(X_2, X_1)=d(X_1, X_2)$, 
in a similar manner we can define the length  $h_{ij}(X_2, X_1)$ of travel path $(A_i, X_2, X_1, A_j)$:  
$$h_{ij}(X_2, X_1) = \min \{ g_{ij}^a(X_2, X_1), g_{ij}^b(X_2, X_1) \}$$
with
$g_{ij}^{t}(X_2, X_1) = ||A_i-X_2||_2+\alpha \, d_t(X_1, X_2)+ ||X_1-A_j||_2$,  
where the superindex $t \in \{a, b\}$ refers to the case determined by
the formula for $d(X_1, X_2)$.

For example,  when
$X_1 \in L_p$, $X_2 \in L_q$,
 Figure 3 (left) shows the function $h_{ij}(X_1, X_2)$ 
of network  of Figure 2 for $\alpha=0{.}3$. Note the
 {\it pagoda roof}-shape of the  surface.

For each O/D pair $(i, j)$, the functions $h_{ij}(X_1, X_2)$ and $h_{ij}(X_2, X_1)$ indicate the length of travel paths $(A_i, X_1, X_2, A_j)$ and $(A_i, X_2, X_1, A_j)$, respectively. Therefore, 
the following notation uses the superindex ``12" for the  first path and ``21" for the second path (this notation was also used in~\cite{KorMesa11}).

\begin{definition}[Sublevel Sets]  
Let $\{ L_p, L_q\}$ be two linear arc segments of type 1 such that
$(X_1, X_2) \in L_p \times L_q$.  
For $\eta \ge 0$, let us consider:
\begin{enumerate}
\item  The  ($\eta$)-sublevel set $S_{ij}^{12}(\eta)$ of  function $h_{ij}(X_1, X_2)$,
given by:
$$S_{ij}^{12}(\eta)= \{ (X_1, X_2) \in L_p \times L_q : h_{ij}(X_1, X_2) \le \eta \}$$
Analogously, $S_{ij}^{21}(\eta)= \{ (X_1, X_2) \in L_p \times L_q : h_{ij}(X_2, X_1) \le \eta \}$.
\item  For $t \in \{a, b\}$, the  ($\eta$)-sublevel sets
$ G_{ij}^{12(t)}(\eta)$ and $G_{ij}^{21(t)}(\eta) $, of functions
$g_{ij}^t(X_1, X_2)$ and
$g_{ij}^t(X_2, X_1)$, respectively,  given by:
$$ \begin{array}{l}
G_{ij}^{12(t)} (\eta)= \{ (X_1, X_2) \in L_p \times L_q :  \, g_{ij}^t(X_1, X_2) \le \eta \} \\[1ex]
G_{ij}^{21(t)} (\eta)= \{ (X_1, X_2) \in L_p \times L_q :  \, g_{ij}^t(X_2, X_1) \le \eta \}
\end{array} $$
\end{enumerate}
\end{definition}

For several $\eta$-values, Figure 3 (right)   displays the corresponding  level sets $S_{ij}^{12} (\eta)$  of the surface on the left of the figure.

\begin{center}
\begin{tabular}{cc}
\scalebox{0.35}[0.35]{\includegraphics{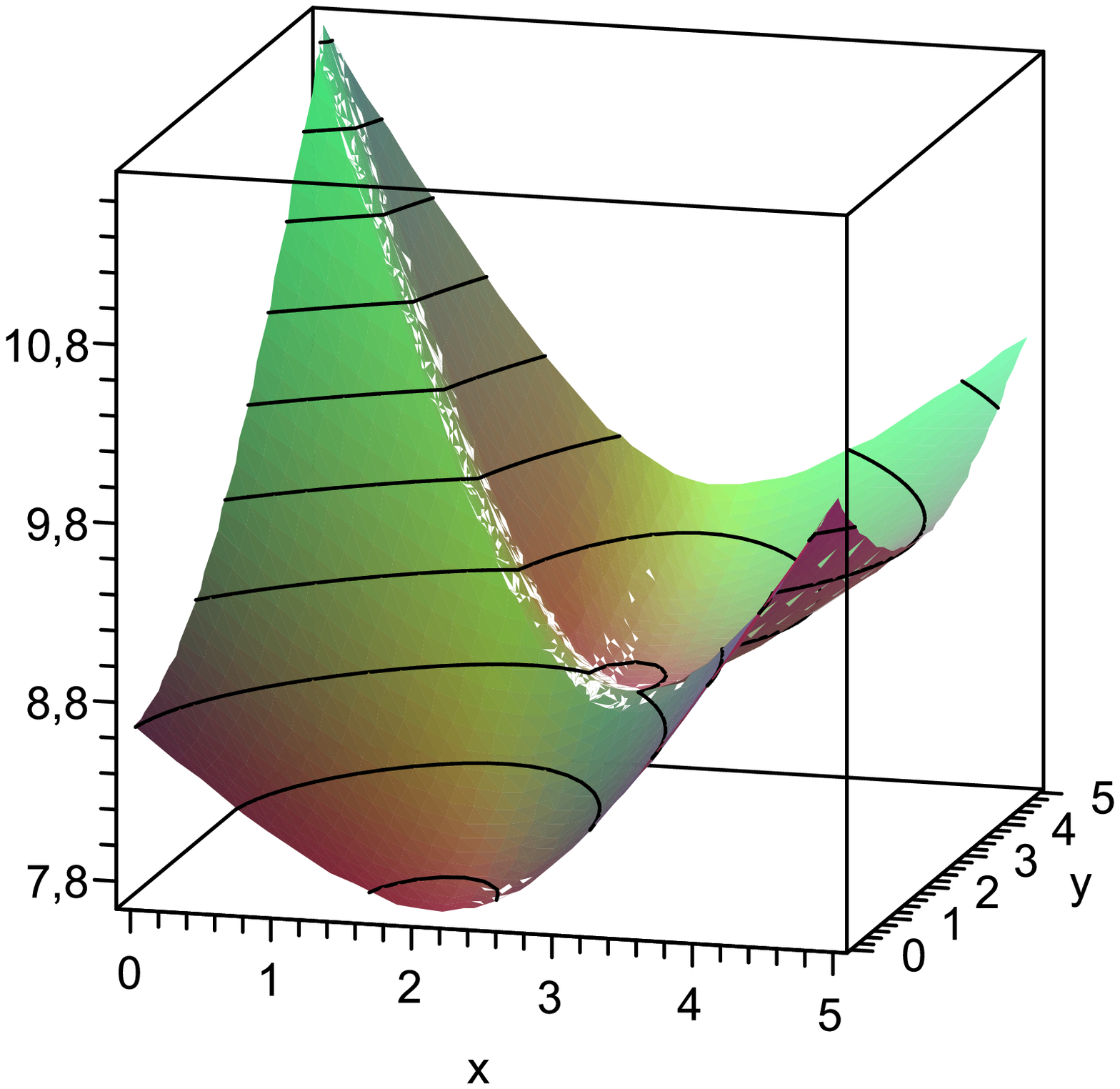}} 
&
\scalebox{0.28}[0.28]{\includegraphics{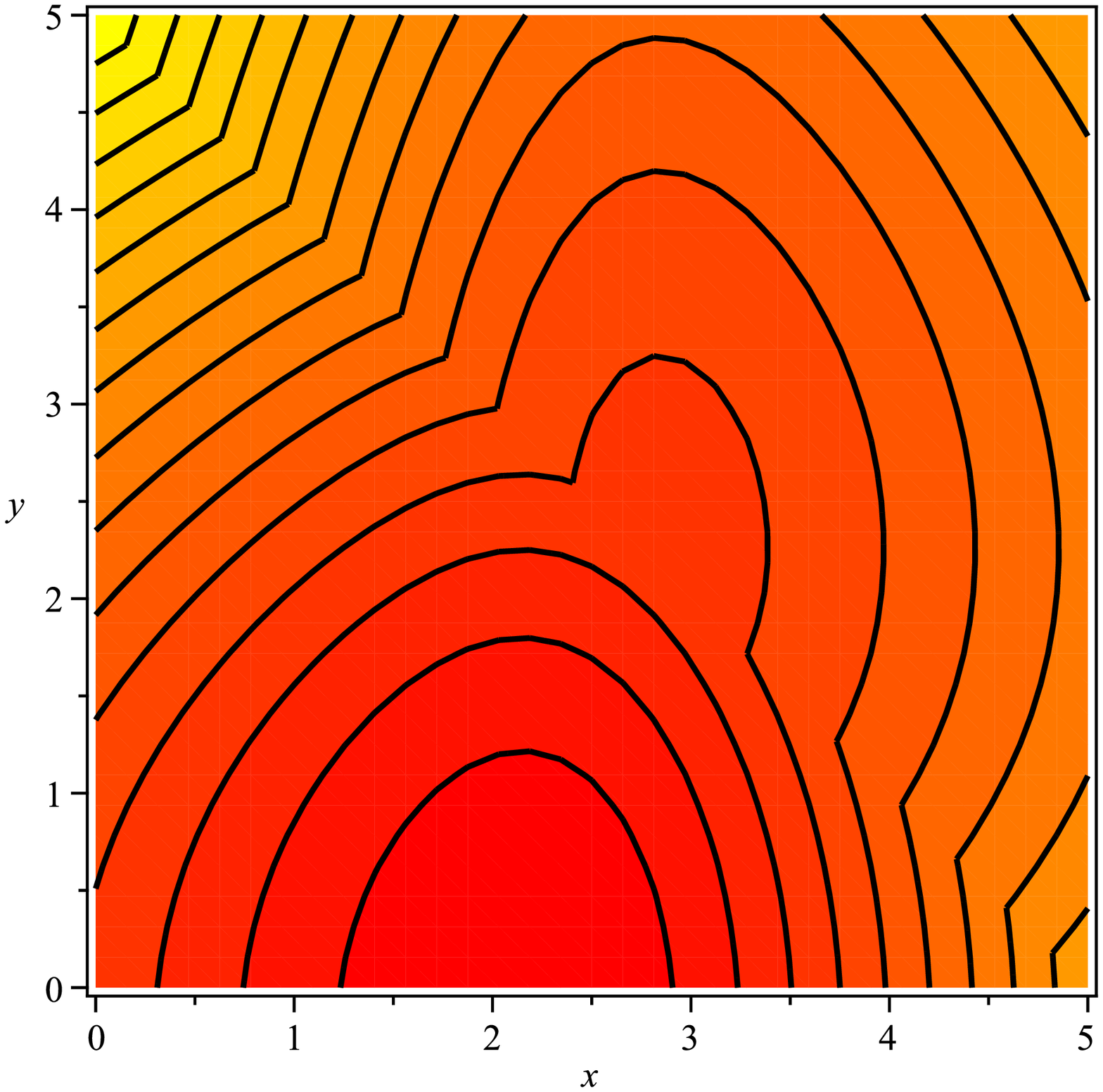}} 
\end{tabular} 
\end{center}
\vskip -1ex
{\centerline{Figure 3.  Surface $h_{ij}(X_1, X_2)$ for  $\alpha=0{.}3$ (left), and 
sets  $S_{ij}^{12}(\eta)$ of such a surface (right)}}
 
\vskip 4ex

For $t \in \{a, b\}$, the functions $g_{ij}^t(X_1, X_2)$ are the sum of the convex  term $||A_i-X_1||_2+||A_j-X_2||_2$
 and the  linear function $\alpha \, d_t (X_1, X_2)$. Similarly, 
 $g_{ij}^t(X_2, X_1)$ is also the sum of a convex and a linear term.  Therefore,  the convexity of the sets  $G_{ij}^{12(t)}(\eta)$ and $G_{ij}^{21(t)}(\eta)$, for $t \in \{a, b\}$,
is a straightforward consequence (see~\cite{Boyd}).

\vskip 2ex
From these definitions, we have 
\begin{lemma} \label{lema4a}
$S_{ij}^{\tau}(\eta) = G_{ij}^{\tau (a)}(\eta) \cup G_{ij}^{\tau (b)}(\eta)$,  for $\tau \in \{12, 21\}$.
\end{lemma}
\begin{proof}
We  prove the lemma by double inclusion.
Without loss of generality, we  assume $\tau=12$.
Let $(X_1, X_2) \in S_{ij}^{12}(\eta)=\{ (X_1, X_2) \in L_p \times L_q: \; h_{ij}(X_1, X_2) \le \eta\}$ be such that $h_{ij}(X_1, X_2) =\min \{g_{ij}^a(X_1, X_2), g_{ij}^b (X_1, X_2)\} = g_{ij}^t (X_1, X_2)$, for some $t \in \{a, b\}$. Since $  g_{ij}^t (X_1, X_2) \le \eta$,
then $(X_1, X_2) \in G_{ij}^{12 (t)}$,  which implies  $S_{ij}^{12}(\eta) \subseteq G_{ij}^{12(a)}(\eta) \cup G_{ij}^{12(b)}(\eta)$.

Conversely, if $(X_1, X_2) \in G_{ij}^{12 (a)}(\eta) \cup G_{ij}^{12(b)}(\eta)$, then 
$(X_1, X_2)$ belongs  to (at least) one of such sets, suppose $(X_1, X_2)  \in G_{ij}^{12 (a)}(\eta) $. If $(X_1, X_2) \notin  G_{ij}^{12(b)}(\eta)$, then $g_{ij}^b (X_1, X_2) > \eta$, therefore $h_{ij}(X_1, X_2) =  g_{ij}^a (X_1, X_2) \le \eta$ which implies
$(X_1, X_2) \in S_{ij}^{12}(\eta)$. If $(X_1, X_2) \in  G_{ij}^{12(b)}(\eta)$, then
$h_{ij} (X_1, X_2) \le \max \{ g_{ij}^a (X_1, X_2), g_{ij}^b (X_1, X_2) \} \le \eta$, consequently $(X_1, X_2) \in S_{ij}^{12}(\eta)$. In this case we have obtained
the converse inclusion
$G_{ij}^{12(a)}(\eta) \cup G_{ij}^{12(b)}(\eta) \subseteq S_{ij}^{12}(\eta)$,
which concludes the proof.
\end{proof}

From this result, both $S_{ij}^{12}(\eta) = G_{ij}^{12(a)}(\eta) \cup G_{ij}^{12(b)}(\eta) $ and $S_{ij}^{21}(\eta) = G_{ij}^{21(a)}(\eta) \cup G_{ij}^{21(b)}(\eta)$ are the union of two convex (but possibility not disjoint) sets.  The following example shows that 
$G_{ij}^{12(a)}(\eta) \cap G_{ij}^{12(b)}(\eta) \neq \emptyset$, for some $\eta$-values, with $\eta < ||A_i-A_j||_2$. 

In  Figure 2, we have $||A_i-A_j||_2 = \D \frac{\sqrt{409}}{2}\simeq 10.11$. Figure 4 (a) displays, for $\alpha=0.4$ and $\eta = 10$, the boundaries of the sublevel sets $G_{ij}^{12(a)}(\eta)$ and $G_{ij}^{12(b)}(\eta)$. It can be observed that the intersection of both sets is not empty.

\bigskip

\begin{proposition}\label{propo5a}
For any $0 \le \eta < ||A_i-A_j||_2$, 
$S_{ij}^{12}(\eta) \cap S_{ij}^{21}(\eta) = \emptyset$.
\end{proposition}
\begin{proof}
Let $\eta $ be any value such that $0 \le \eta < ||A_i-A_j||_2$. To establish the result it is sufficient to prove that, for $t \in \{a, b\}$, 
$G_{ij}^{12(t)}(\eta) \cap S_{ij}^{21}(\eta) = \emptyset$. Without loss of generality, we will prove $G_{ij}^{12(a)}(\eta) \cap S_{ij}^{21}(\eta) = \emptyset$. 

Assume $(X_1, X_2) \in G_{ij}^{12(a)} (\eta)$. Then
$$g_{ij}^a(X_1, X_2)=||A_i-X_1||_2+\alpha \, d_a(X_1, X_2) + ||X_2-A_j||_2 \le
\eta < ||A_i-A_j||_2$$
This implies $||A_i-X_1||_2+||X_2-A_j||_2<||A_i-A_j||_2$. By combining this inequality with the triangular property, we can write
$$||A_i-A_j||_2 \le ||A_i-X_1||_2+||X_1-A_j||_2< ||A_i-A_j||_2-||X_2-A_j||_2+
||X_1-A_j||_2$$
By applying again the triangular property $||A_i-A_j||_2 \le ||A_i-X_2||_2+||X_2-A_j||_2$ to the right hand side, we finally obtain
$||A_i-A_j||_2 < ||A_i-X_2||_2+||X_1-A_j||_2$. Consequently
$$\eta < ||A_i-A_j||_2 < ||A_i-X_2||_2+||X_1-A_j||_2 + \alpha \min \{ d_a(X_1, X_2), 
d_b(X_1, X_2) \} = h_{ij}(X_2, X_1)$$
Since $\eta <  h_{ij}(X_2, X_1)= \min \{ g_{ij}^a(X_2, X_1), g_{ij}^b(X_2, X_1) \}$, then $(X_1, X_2) \notin G_{ij}^{21(a)}(\eta) \cup G_{ij}^{21(b)}(\eta) = S_{ij}^{21}(\eta)$, which concludes the proof. 
\end{proof}

Let $\{L_p, L_q\}$ be two linear arc segment of type 1.
For a given point $(X_1, X_2) \in L_p \times L_q$, the objective value 
$F(X_1, X_2)= \D \sum_{(i, j) \in C(X_1, X_2)} \, t_{ij}$ quantifies the amount of trip patterns captured by such a point.  Taking into account  definition~\ref{cover}, let $S_{ij} $ be the set of points which cover the  O/D pair $(i, j)$, given by
$$ S_{ij} = \{ (X_1, X_2) \in L_p \times L_q : \, f_{ij}(X_1, X_2) \le d_{ij} \}$$
The following result follows directly from the definition:
\begin{corollary} \label{coro11}
$(X_1, X_2) \in S_{ij}$ if and only if $(i, j) \in C(X_1, X_2)$
\end{corollary}

In order to obtain $S_{ij}$ from the above defined sublevel sets, 
we replace the $\eta$-values by the specific acceptance level 
$d_{ij}$ associated with each O/D pair $(i, j)$.

\begin{remark}[Notation]
Given an acceptance level $0 \le d_{ij}< ||A_i-A_j||_2$,
 for $t \in \{a, b\}$ let us denote $G_{ij}^{12(t)}(d_{ij})$ and $G_{ij}^{21(t)}(d_{ij})$ by $G_{ij}^{12(t)}$ and $G_{ij}^{21(t)}$, respectively. This abbreviated notation  is also applied to the  sets
$S_{ij}^{12}(d_{ij})$ and  $S_{ij}^{21}(d_{ij})$, that is:
$S_{ij}^{12} = S_{ij}^{12}(d_{ij})$, and $S_{ij}^{21}=S_{ij}^{21}(d_{ij})$.
\end{remark}

\begin{corollary} \label{coro13}
For each O/D pair $(i, j)$, \ $S_{ij} = S_{ij}^{12} \cup S_{ij}^{21}$.
\end{corollary}
\begin{proof}
Since $f_{ij}(X_1, X_2) = \min \{ h_{ij}(X_1, X_2), h_{ij} (X_2, X_1) \}$, the result follows  straightforwardly.
\end{proof}

\begin{note}
We now briefly comment on the role of $S_{ij}$ in the construction of the FDS.
Let us suppose that $C(X_1, X_2)=\{ (i, j), \ldots, (k, r)\}$ is the set of O/D pairs covered by $(X_1, X_2)$. Corollary~\ref{coro11} implies that
$(X_1, X_2) \in S_{ij} \cap \ldots \cap S_{kr}$. For this reason, we next analyze the intersections of these sets, since such intersections will provide the points  to be incorporated in the FDS. In fact, due to the absence of convexity of both
$S_{ij}^{12}$ and $S_{ij}^{21}$, we will focus the effort on 
selecting points from the boundaries of all these sets (as well as from their intersections),
in order to guarantee that the selected points belong to all sets involved in the intersection.
\end{note}

Henceforth we will use  the notation $\overline{G}$ (or $\overline{S}$), to identify the boundary curve of a set $G$ (or $S$). That is, for $t \in \{a, b\}$, we have
$$ \begin{array}{c}
\overline{G}_{ij}^{12(t)} = \{ (X_1, X_2) \in L_p \times L_q :  \, g_{ij}^t(X_1, X_2) = d_{ij} \}  \\[1ex]
\overline{S}_{ij}^{12} = \{ (X_1, X_2) \in L_p, L_q : \, h_{ij}(X_1, X_2) = d_{ij} \}
 \end{array} $$
and similarly for $\overline{G}_{ij}^{21(t)}$ and $ \overline{S}_{ij}^{21}$. Clearly, for $\tau \in \{12, 21\}$, 
 $\overline{S}_{ij}^{\, \tau} \subseteq \overline{G}_{ij}^{\, \tau (a)} \cup \overline{G}_{ij}^{\, \tau (b)}$, and  $\overline{S}_{ij}^{\, \tau} = \overline{G}_{ij}^{\, \tau (a)} \cup 
\overline{G}_{ij}^{\, \tau (b)}$
if and only if the set $G_{ij}^{\, \tau (a)} \cap G_{ij}^{\, \tau (b)} $ contains at most one point.  Likewise
$$ \overline{S}_{ij} = \{ (X_1, X_2) \in L_p \times L_q : \, f_{ij} (X_1, X_2) = d_{ij} \}$$ 
Since from Proposition~\ref{propo5a}, $S_{ij}^{12} \cap S_{ij}^{21} = \emptyset$, then
$\overline{S}_{ij}=\overline{S}_{ij}^{12} \cup \overline{S}_{ij}^{21}$.

\vskip 1ex

As we have already seen, in order to construct the FDS of problem (3), we will  successively selecting several feasible points from both these boundary curves and their intersections.

\begin{lemma}\label{lema4b}
Given two different O/D pairs $(i, j)$, $(k, r)$, let ${\cal P}(ij; kr) \subset L_p \times L_q  $ be the set of intersection points
 defined as follows:
$$ {\cal P}(ij; kr) = \bigcup \left\{  (\overline{G}_{ij}^{\, \tau (t)} \cap \overline{G}_{kr}^{\, \tau' (t')}), \; \tau, \tau' \in \{12, 21\}, \; 
t, t' \in \{a, b\}  \right\} $$
Then, $\overline{S}_{ij} \cap \overline{S}_{kr} \subseteq {\cal P}(ij; kr) $.
\end{lemma}
\begin{proof}
 $\overline{S}_{ij} = 
\overline{S}_{ij}^{12} \cup \overline{S}_{ij}^{21} \subseteq
\left( \overline{G}_{ij}^{\, 12 (a)} \cup \overline{G}_{ij}^{\, 12 (b)} \right) \cup
\left( \overline{G}_{ij}^{\, 21 (a)} \cup \overline{G}_{ij}^{\, 21 (b)} \right)$, and a similar  inclusion can be obtained for $\overline{S}_{kr}$. Therefore we can write 
$$\overline{S}_{ij} \cap \overline{S}_{kr} \subseteq 
\left( \overline{G}_{ij}^{\, 12 (a)} \cup \overline{G}_{ij}^{\, 12 (b)} \cup
\overline{G}_{ij}^{\, 21 (a)} \cup \overline{G}_{ij}^{\, 21 (b)} \right)
\bigcap 
\left( \overline{G}_{kr}^{\, 12 (a)} \cup \overline{G}_{kr}^{\, 12 (b)} \cup
\overline{G}_{kr}^{\, 21 (a)} \cup \overline{G}_{kr}^{\, 21 (b)} \right) $$
Clearly, ${\cal P}(ij; kr) $ is right side of this expression  since
${\cal P}(ij; kr) $ is obtained by intersecting each set of the first group with each set of the second group. This concludes the proof. 
\end{proof}
 Figure 4   shows the sets $ {\cal P}(ij; kr)$  obtained in several cases.
For the same points $A_i(2{.}5, 6)$ and $A_j(1, -4)$ of Figure 2,  and the value $\alpha=0{.}4$, Figure 4 (a) shows the curves $\overline{G}_{ij}^{\, 12 (a)}$ and $\overline{G}_{ij}^{\, 12 (b)}$ obtained for the acceptance level $d_{ij}=10$ (the set $S_{ij}^{21}$ is empty). Likewise,  we have  considered two new points:
$A_k(-2, -4{.}5)$ and $A_r(3, 5{.}5)$, with $||A_k-A_r||\simeq 11{.}18$. 
 Figure 4 (b) and (c) display the set of intersection points $ {\cal P}(ij; kr)$ obtained by considering different acceptance levels for the pair $(k, r)$. Thus, for  $d_{kr}= 10{.}5$, the set $ {\cal P}(ij; kr)$ (Figure 4 (b)) is non-empty, and contains three intersection points, and for $d_{kr}= 9{.}8$ (Figure 4 (c)),  such a set is empty. 
 Moreover, in this last figure only a branch of curve $\overline{G}_{kr}^{\, 21(a)}$ is inside the feasible domain. Note that, in all cases, the points of $ {\cal P}(ij; kr)$ are intersections of all pairs of curves $\{ \overline{G}_{ij}^{(\cdot)}, \overline{G}_{kr}^{(\cdot)} \}$.

\bigskip

\begin{center}
\begin{tabular}{ccc}
\scalebox{0.2}[0.2]{\includegraphics{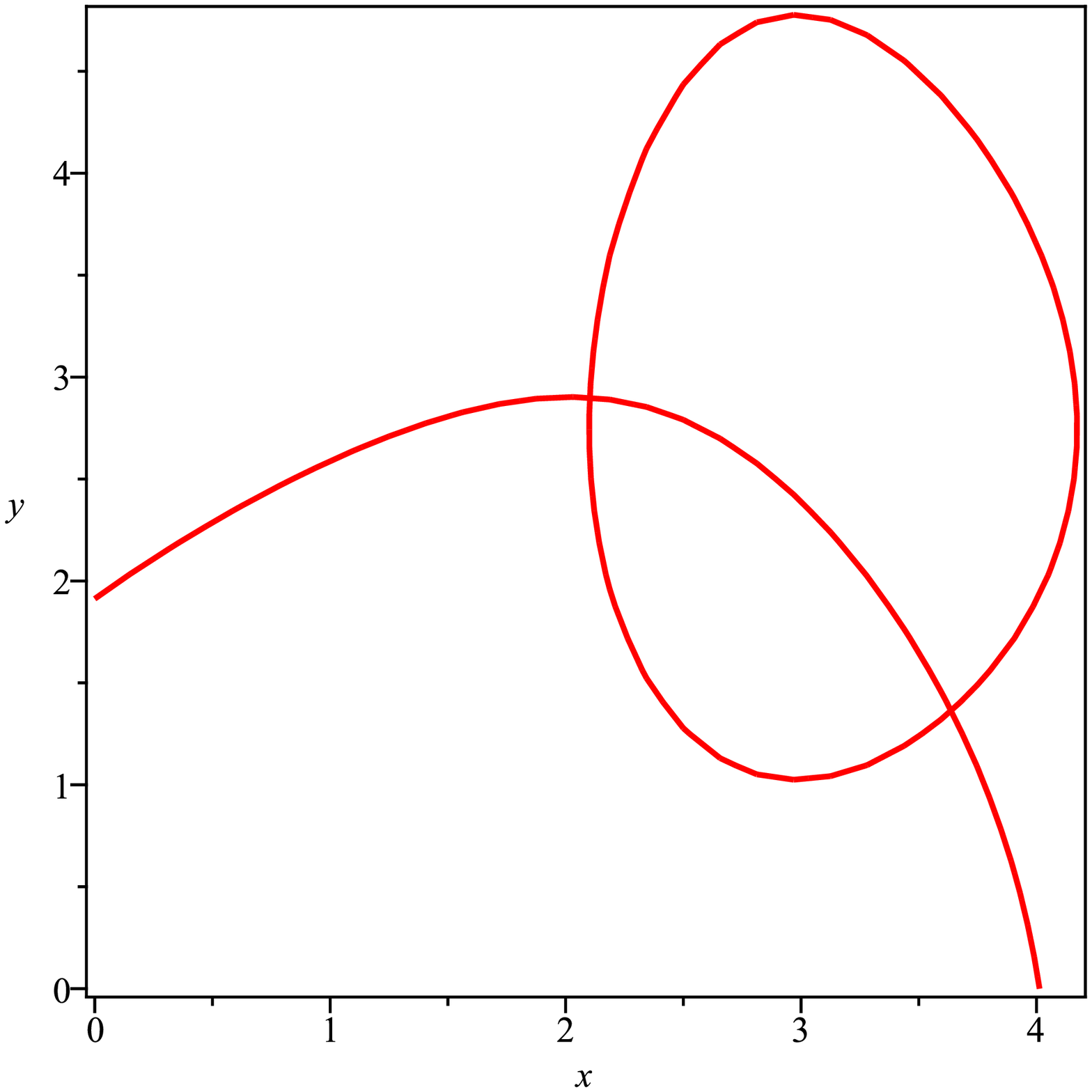}}
&
 \scalebox{0.2}[0.2]{\includegraphics{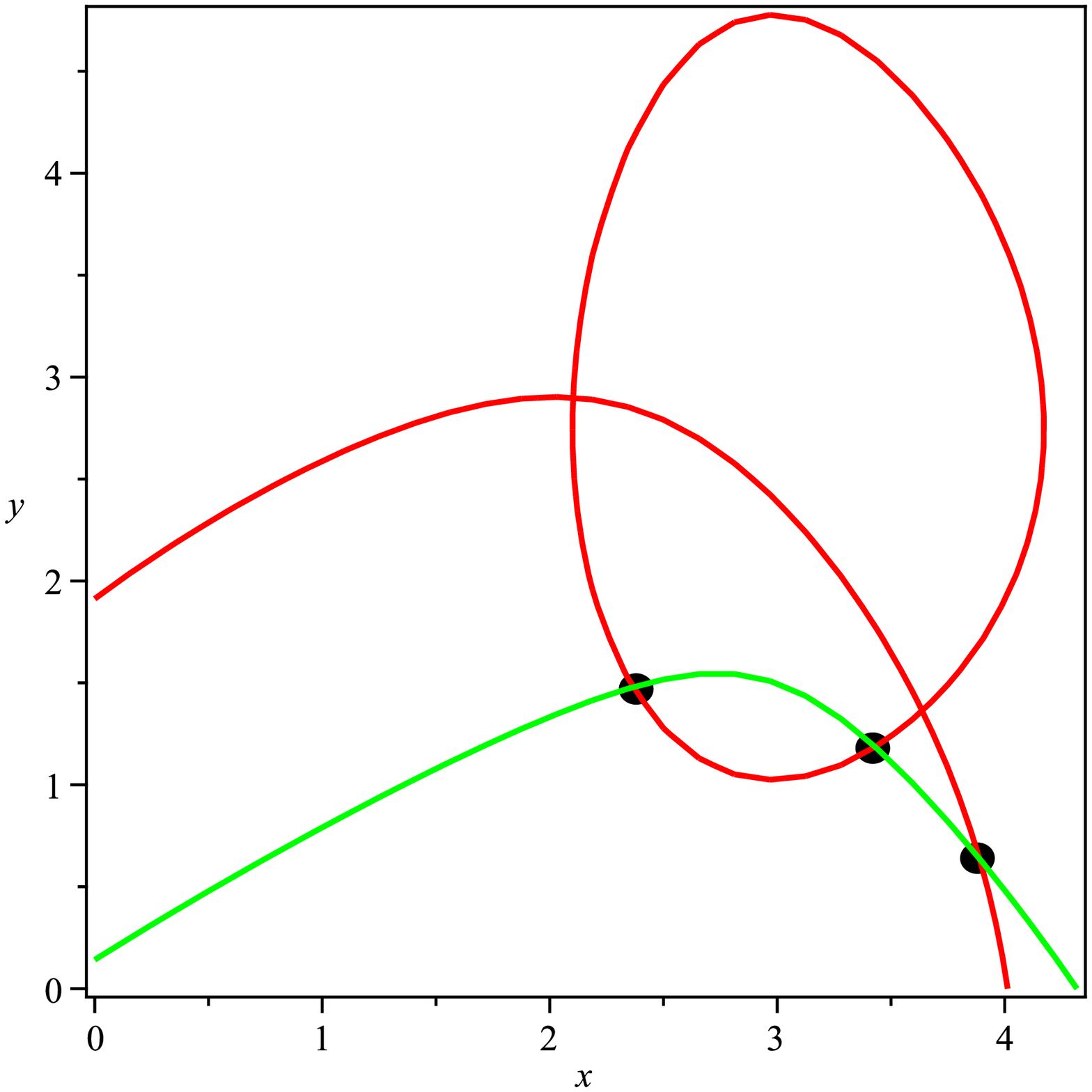}}
&
\scalebox{0.2}[0.2]{\includegraphics{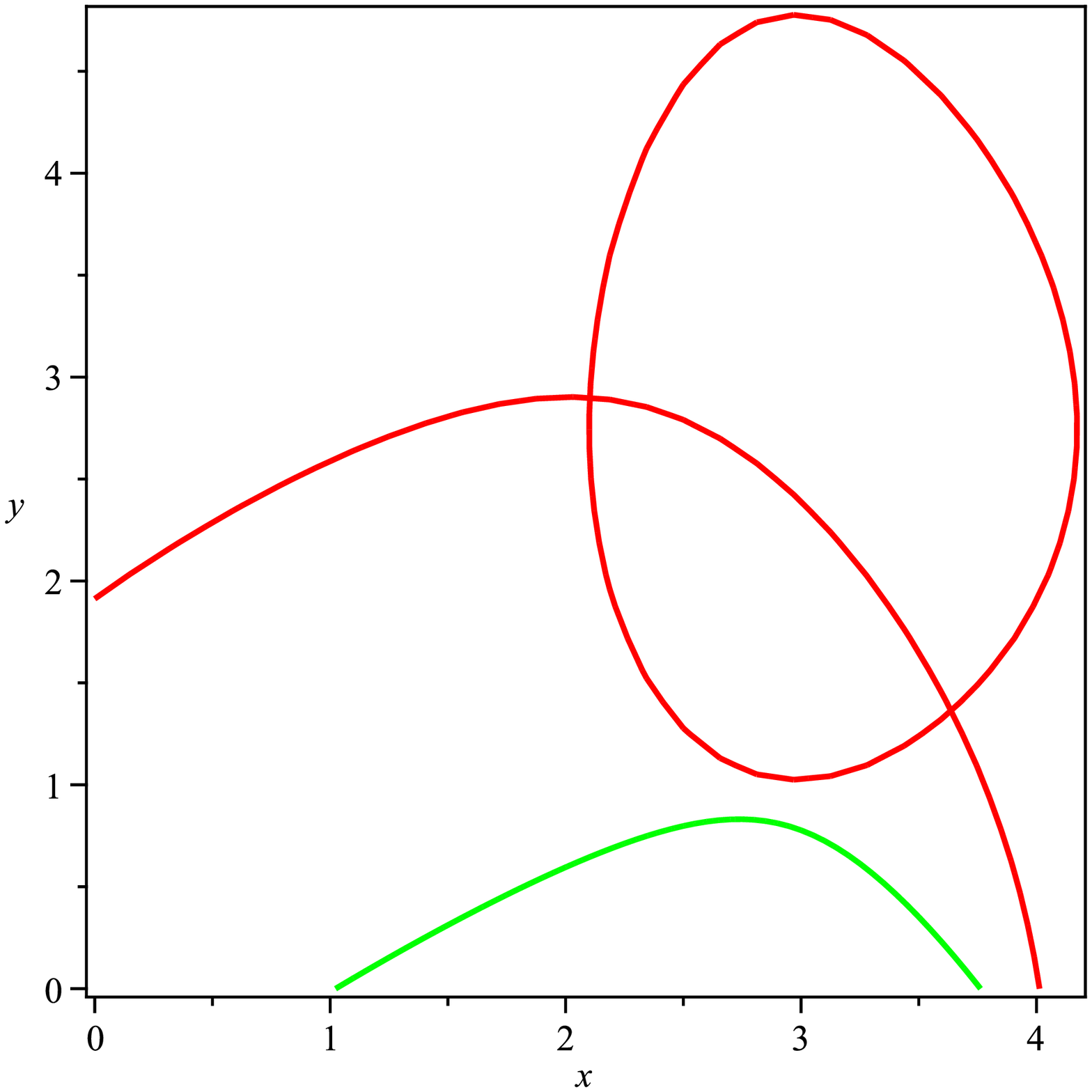}} \\
(a) & (b) & (c) 
\end{tabular}
\end{center}
Figure 4. (a): Curves  $\overline{G}_{ij}^{\, 12 (a)}$ and $\overline{G}_{ij}^{\, 12 (b)}$, for $d_{ij}=10$. (b): Set $ {\cal P}(ij; kr)$, for $d_{kr}=10{.}5$.
(c) For $d_{kr}=9{.}8$, there are no intersection points and  $ {\cal P}(ij; kr)=\emptyset$.

\bigskip

\begin{definition} \label{defi15}
For each O/D pair $(i, j)$, 
let  ${\cal Q} (i, j) = {\cal Q}^{12}(i,j) \cup {\cal Q}^{21}(i,j)$ be a set of feasible points
 of $\overline{S}_{ij}=\overline{S}_{ij}^{12} \cup \overline{S}_{ij}^{21} $, 
where for each $\tau \in \{12, 21\}$:
$$ {\cal Q}^\tau(i, j) = \left\{ \begin{array}{cl}
\overline{G}_{ij}^{\,\tau(a)} \cap \overline{G}_{ij}^{\,\tau(b)}, & \mbox{
if this intersection contains at least one feasible point} \\[0.2ex]
\{ Y_a, Y_b  \}, &  \begin{array}{l} \mbox{} \\[0.6ex]
\mbox{where $Y_a$ and $Y_b$ are feasible points  arbitrarily selected} \\
\mbox{from $\overline{G}_{ij}^{\,\tau(a)}$  
and $\overline{G}_{ij}^{\,\tau(b)}$, respectively} \end{array}
\end{array} \right. 
$$
\noindent
Clearly ${\cal Q}^\tau(i, j) \subset \overline{S}_{ij}^{\, \tau}$, for  $\tau \in \{12, 21\}$.  
 \end{definition}

\begin{theorem} \label{teo17}
For  the restricted problem (3), let ${\cal P} \subset L_p \times L_q$ be the set defined as follows
$$  \quad {\cal P} = \bigcup_{(i, j) \neq (k, r)} {\cal P}(ij; kr) $$
Then, at least one of the following three cases occurs:
\begin{enumerate}
\item Any point of $L_p \times L_q$ is an optimal  solution for problem (3).
\item There exists an O/D pair $(i, j)$ such that ${\cal Q}(i, j) $ contains (at least) one optimal solution for problem (3).
\item The set ${\cal P}$ contains (at least) one optimal solution for problem (3).
\end{enumerate} 

\end{theorem}
\begin{proof}
Let $(\widehat{X}_1, \widehat{X}_2) \in L_p \times L_q$ and $F(\widehat{X}_1, \widehat{X}_2)$  be an optimal solution for problem (3) and the corresponding  optimal value of the objective function, respectively. That is, $F(\widehat{X}_1, \widehat{X}_2) \ge F(X_1, X_2)$, $\forall (X_1, X_2) \in L_p \times L_q$. Let $C(\widehat{X}_1, \widehat{X}_2)$ be the set of O/D pairs covered by $(\widehat{X}_1, \widehat{X}_2)$
\begin{enumerate}
\item The trivial  case $C(\widehat{X}_1, \widehat{X}_2) = \emptyset$ implies that 
$F(\widehat{X}_1, \widehat{X}_2)=0$, therefore all points of $L_p \times L_q$ are optimal.

 \item  If $C(\widehat{X}_1, \widehat{X}_2) \neq \emptyset$, then 
$(\widehat{X}_1, \widehat{X}_2)$ covers at least one O/D pair.  Since from Corollary~\ref{coro11}, $\forall
(i, j) \in C(\widehat{X}_1, \widehat{X}_2)$, $(\widehat{X}_1, \widehat{X}_2) \in S_{ij}$, we can write
$$(\widehat{X}_1, \widehat{X}_2) \in I(\widehat{X}_1, \widehat{X}_2) :=
\D \bigcap_{ (i, j) \in 
C(\widehat{X}_1, \widehat{X}_2)  }  S_{i j} $$
From the construction,  $I(\widehat{X}_1, \widehat{X}_2)$ is the set of points which cover the O/D pairs of $C(\widehat{X}_1, \widehat{X}_2)$. To prove the statements 2 and 3 of  Theorem (\ref{teo17}),  we analyze separately the cases $|C(\widehat{X}_1, \widehat{X}_2)| =1$ and
$|C(\widehat{X}_1, \widehat{X}_2)| >  1$.

\begin{enumerate}

\item If $C(\widehat{X}_1, \widehat{X}_2)= \{ (i, j)\}$, then 
$I(\widehat{X}_1, \widehat{X}_2) \subseteq S_{ij}$, and $I(\widehat{X}_1, \widehat{X}_2) \cap  S_{kr} = \emptyset$, $\forall (k, r) \neq (i, j)$. Therefore
at least one of the sets
$S_{ij}^{12}$, $S_{ij}^{21}$ is contained into $I(\widehat{X}_1, \widehat{X}_2)$.

 If both sets are contained in $I(\widehat{X}_1, \widehat{X}_2)$,  then
 $(\widehat{X}_1, \widehat{X}_2) \in I(\widehat{X}_1, \widehat{X}_2)=S_{ij}=S_{ij}^{12} \cup S_{ij}^{21}$.  Since they are disjoint sets,  $(\widehat{X}_1, \widehat{X}_2)$ is contained in one of them, suppose this set is
$S_{ij}^{12}$. Moreover, 
all points of $S_{ij}^{12}$ cover the same pairs as $(\widehat{X}_1, \widehat{X}_2)$ 
(otherwise $(\widehat{X}_1, \widehat{X}_2)$ would not be optimal), 
 i.e.  $F(\widehat{X}_1, \widehat{X}_2)=F(X_1, X_2)$, $\forall (X_1, X_2) \in S_{ij}^{12}$.  Consequently  all points of this set (including those on its boundary) are optimal. 
Both in case 
$G_{ij}^{12(a)}$ and $G_{ij}^{12(b)}$  are disjoint or not, we have
$ \overline{S}_{ij}^{12} \cap {\cal Q}^{12}(i, j) \neq \emptyset$, therefore at least one point of ${\cal Q}(i, j)$ is also an optimal solution of (3).

 If only one of the sets  is entirely contained in $I(\widehat{X}_1, \widehat{X}_2)$, suppose this set is $S_{ij}^{12}$, then $I(\widehat{X}_1, \widehat{X}_2) \cap S_{ij}^{21} = \emptyset$, and the previous reasoning can be also repeated for this situation.  This proves the second assertion.

Note that, in this case  it is not possible to have  
$(\widehat{X}_1, \widehat{X}_2) \in I(\widehat{X}_1, \widehat{X}_2) \cap S_{ij}^{21}
\subset S_{ij}^{21}$, since this would imply that
 $I(\widehat{X}_1, \widehat{X}_2) \cap S_{ij}^{21}$ is obtained from the intersection of $S_{ij}^{21}$ with (at least) one other set
 $S_{kr}$, with $(k, r) \in C(\widehat{X}_1, \widehat{X}_2)$, which is a contradiction.

\item 
Suppose  $|C(\widehat{X}_1, \widehat{X}_2) | > 1$. Since each set $S_{ij}$ involved in the intersection $I(\widehat{X}_1, \widehat{X}_2) $ is the union of two disjoint and non-convex sets,   $I(\widehat{X}_1, \widehat{X}_2) $ could consist of several disjoint subsets, and the boundary of $I(\widehat{X}_1, \widehat{X}_2)$  is obtained from pieces of boundaries of the collection $\{ \overline{S}_{ij}^{\tau}, \, (i, j) \in C(\widehat{X}_1, \widehat{X}_2), \, \tau \in \{12, 21\}  \}$.

 The optimum $(\widehat{X}_1, \widehat{X}_2)$ belongs to one of the subsets composing
$ I(\widehat{X}_1, \widehat{X}_2)$, let  $\widehat{S} \subseteq I(\widehat{X}_1, \widehat{X}_2)$ denote the maximal connected  subset containing $ (\widehat{X}_1, \widehat{X}_2)$. That is, if for instance $(i, j), (k, r) \in C(\widehat{X}_1, \widehat{X}_2)$, then
$I(\widehat{X}_1, \widehat{X}_2) \subseteq (S_{ij}^{12} \cap S_{kr}^{12}) \cup  (S_{ij}^{12} \cap S_{kr}^{21})  \cup
(S_{ij}^{21} \cap S_{kr}^{12}) \cup (S_{ij}^{21} \cap S_{kr}^{21})$, consequently 
$\widehat{S}$ would be  contained in one of these four sets. All points of $\widehat{S}$ are also optimal, and the boundary of $\widehat{S}$ is composed by pieces of boundaries of the collection   
 above described. Then the boundary of $\widehat{S} $ contains at least one intersection point of two curves  $\overline{S}_{ij}^{\tau}$,  $\overline{S}_{kr}^{\tau'}$ composing  the boundary of  $I(\widehat{X}_1, \widehat{X}_2)$
 (with $\tau, \tau' \in \{12, 21\}$).  From Lemma~\ref{lema4b}, the set ${\cal P}(ij;kr)$ contains all intersection points of   
$\overline{S}_{ij}^{\tau} \cap \overline{S}_{kr}^{\tau'}$. Therefore, 
$\widehat{S} \cap {\cal P} \neq \emptyset$, which is the third statement of the theorem.  This concludes the proof.

\end{enumerate}
\end{enumerate}

\end{proof}

\begin{corollary} \label{coro18}
Let ${\cal Q} \subset L_p \times L_q$ be the set defined as
${\cal Q} = \D \bigcup_{(i, j)} \, {\cal Q} (i, j) $, and let $X_{pq} $ be an arbitrary point selected from  $L_p \times L_q$
 such that $X_{pq} \notin {\cal Q} \cup {\cal P}$. 
Then the set
${\cal Q} \cup {\cal P} \cup \{ X_{pq} \}$ is a $\mbox{\rm FDS}$  for the restricted problem (3).
\end{corollary}

Saving some technical questions dealing with the bi-dimensional representation of  Problem (3) by means of a parametrization of linear arc segments,  the algorithm for solving (3) is based on progressively constructing the FDS, and then evaluating the objective function $F$ on the points of the FDS for finally selecting the best solution. 
At the end of  the preprocessing phase aforementioned in which the ordered sequence 
${\cal L}(e)$ of all linear arc segments of each edge $e$ is computed, we can know if two linear arc segments $L_p \in {\cal L}(e_p)$ and  $L_q \in {\cal L}(e_q)$ are, or not, antipodal to each other. In case of having a pair $\{ L_p, L_q\}$ of type 1, the distance 
$d(X_1, X_2)$, for $(X_1, X_2) \in L_p \times L_q$ can be established from
Lemmas~\ref{lema6} and~\ref{lema7} in constant time, without increasing the complexity of the previous phase. Finally we point out that, in the description of the algorithm, all points of the FDS are incorporated to the set $\Omega$.

\vskip 2ex
\noindent Algorithm for problem (3)
\begin{enumerate}
\item Set  ${\Omega}:= \emptyset$.
\item For each O/D pair $(i, j)$, {\bf do}
\begin{enumerate}
\item  Obtain the level curves
$\overline{G}_{ij}^{\, 12 (a)}$, $\overline{G}_{ij}^{\, 12 (b)}$, 
$\overline{G}_{ij}^{\, 21 (a)}$, and $\overline{G}_{ij}^{\, 21 (b)}$.
\item  For $\tau \in \{12, 21 \}$, compute the intersection set $ {\cal Q}^{\tau}(i, j)=\overline{G}_{ij}^{\, \tau (a)} \cap \overline{G}_{ij}^{\, \tau (b)}$. 
\begin{enumerate}
     \item  If $ {\cal Q}^{\tau}(i, j) \neq \emptyset$, then $\Omega:= \Omega
\cup {\cal Q}^{\tau} (i, j) $. 
 \item Otherwise, set $\Omega:= \Omega \cup \{ Y_a, Y_b\}$, where 
$Y_a$ and $Y_b$ are arbitrary points selected from $\overline{G}_{ij}^{\, \tau (a)}$ and $\overline{G}_{ij}^{\, \tau (b)}$, respectively.
 \end{enumerate}
\end{enumerate}
\item For each different O/D pairs $(i, j)$ and $(k, r)$, {\bf do}
\begin{enumerate}
\item Obtain ${\cal P}(ij;kr)$, the set of intersection points of all pairs of level curves $\{ \overline{G}_{ij}^{\, \tau (t)}, \, \overline{G}_{kr}^{\, \tau' (t')} \}$, for $\tau, \tau' \in \{ 12, 21 \}$ and $t, t' \in \{a, b\}$.
 \item Set ${\Omega}:= {\Omega} \cup {\cal P}(ij;kr)$. 
\end{enumerate}
\item Choose any point $X_{pq} \in L_p \times L_q$ such that $X_{pq} \notin {\Omega}$.  Set $\Omega:=\Omega \cup \{ X_{pq} \}$. 
\item Evaluate the objective function at each point of $\Omega$, and 
select as solution a point $(X_1^{*}, X_2^{*}) $ for which
$F( X_1^{*}, X_2^{*})=\D \max_{(X_1, X_2) \in \Omega} \, F(X_1, X_2)$
(there can be several solutions) 
\end{enumerate}

To discuss the resulting  complexity of this algorithm we first analyze the cardinality of the sets  composing the FDS.
We previously point out that, as it is usual in origin-destination problems, henceforth we will use  $N = O(n^2)$ to denote the number of O/D pairs in order to reflect the complexity as a function of $N$ (instead of as a function of the number $n$ of isolated facilities).

In this process, the main computational effort  is spent  in computing the $O(N^2)$ sets of intersection points ${\cal P}(ij; kr)$.
Each $\overline{G}_{ij}^{\, \tau (t)}$ is the boundary curve of the convex set
$G_{ij}^{\, \tau (t)}$, which is the sublevel set of the convex  function $g_{ij}^{t}(\cdot)$. 
The algebraic structure of these functions allows to establish (by means of Bezout's theorem),  that
for $\tau, \tau' \in \{12, 21\}$ and $t, t' \in \{a, b\}$,  the number of intersection points of two different boundary curves $\overline{G}_{ij}^{\tau (t)}$,
$\overline{G}_{kr}^{\tau' (t')}$ is (upper) bounded by 12 (see~\cite{Kir, KorMesa11}).
Since each ${\cal P}(ij; kr)$ is obtained from the intersection of (at most) 16 pairs of curves, the complexity of ${\cal P}$ (the  number of all intersection points computed in Step 3) is $O(N^2)$. 
From the same argument, the cardinality of ${\cal Q}$ (computed in step 2) is also $O(N^2)$, consequently  $|\Omega| \in O(N^2)$.  Moreover, evaluating the objective function $F$ in each point of $\Omega$ requires $O(N)$ time, which gives a final complexity of $O( N^3)$ for solving the restricted problem (3).

 \subsection{The convex case: The restricted problem of type 2}
We now study the restricted problem (4), given by
$$    \begin{array}{ll} \max & F(X_1, X_2) := \D \sum_{(i, j) \in C(X_1, X_2)} t_{ij} \\[4ex]
\mbox{s.t.} &  (X_1, X_2) \in L_p \times L_q   \\[1ex]
\mbox{} &  L_p, L_q \in {\cal L}, \; \{L_p, L_q\} \; \mbox{of type 2}
\end{array}  \eqno(4)$$
When $(X_1, X_2) \in L_p \times L_q$,   and $\{ L_p, L_q\}$ are two linear arc segments of type 2, from lemmas~{\ref{lema6} and \ref{lema7}} and by replacing   $P $ and $Q $  by $X_1$ and $X_2$, respectively,
 the distance $d(X_1, X_2)$ is either linear or convex (according with $L_p, L_q$ lye in different edges or in the same edge). More specifically, and following with the notation used for problem (3), in this case we have
$$ d_a((X_1, X_2) = d_b (X_1, X_2)  = d(X_1, X_2) $$
Consequently, all results of previous section are valid for problem (4) tacking into account that we now have:
$$S_{ij}^{12}=G_{ij}^{12(a)}=G_{ij}^{12(b)},
\; \mbox{and} \; S_{ij}^{21}=G_{ij}^{21(a)}=G_{ij}^{21(b)}$$
Therefore, from Proposition~\ref{propo5a} and Corollary~\ref{coro13},
 both $S_{ij}^{12}$ and $S_{ij}^{21}$ are convex and disjoint sets such that
$S_{ij}=S_{ij}^{12} \cup S_{ij}^{21}$. Likewise, for this problem the set  ${\cal Q} (i, j)$ of Definition~\ref{defi15} becomes on the set ${\cal Q}(i, j)= \{ Y, Y' \}$, where $Y$ and $Y'$ are points arbitrarily selected from $\overline{S}_{ij}^{12}$ and $\overline{S}_{ij}^{21}$, respectively. And for each two different O/D pairs $(i, j), (k, r)$, the set ${\cal P}(ij; kr)$
of Lemma~\ref{lema4b} is now given by
$$ {\cal P}(ij;kr) = \D \bigcup_{(i, j) \neq (k, r)}  (  \overline{S}_{ij} \cap \overline{S}_{kr} ) =
\{ \overline{S}_{ij}^{12} \cap \overline{S}_{kr}^{12} \}
 \cup  \{ \overline{S}_{ij}^{12} \cap \overline{S}_{kr}^{21} \} \cup
\{ \overline{S}_{ij}^{21} \cap \overline{S}_{kr}^{12} \}
 \cup \{ \overline{S}_{ij}^{21} \cap \overline{S}_{kr}^{21}\}
$$
With these sets Theorem~\ref{teo17} and Corollary~\ref{coro18} are also valid for problem (4) and establishes that ${\cal Q} \cup {\cal P} \cup \{X_{pq} \} $  is the FDS for this problem. Besides, by introducing some slight modifications, the previous Algorithm for problem (3) can be applied for solving problem (4). These modifications consist on replacing steps 2 and 3 for the following ones:

\begin{enumerate}
\setcounter{enumi}{1}
\item For each O/D pair $(i, j)$, {\bf do}
\begin{enumerate}
\item Obtain the level curves $\overline{S}_{ij}^{12}$ and $\overline{S}_{ij}^{21}$
\end{enumerate}
\item For each different  O/D pairs $(i, j)$ and $(k, r)$, {\bf do}
\begin{enumerate}
\item For each $\tau, \tau' \in \{ 12, 21 \}$, compute $\overline{S}_{ij}^{\tau} \cap \overline{S}_{kr}^{\tau'}$.
\begin{enumerate}
\item If $\overline{S}_{ij}^{\tau} \cap \overline{S}_{kr}^{\tau'} \neq \emptyset$, then
set $\Omega:= \Omega \cup \overline{S}_{ij}^{\tau} \cap \overline{S}_{kr}^{\tau'}$
\item Otherwise, set $\Omega:= \Omega \cup \{ Y_{ij}, Y_{kr} \}$, where
$Y_{ij}, Y_{kr}$ are points arbitrarily selected from $\overline{S}_{ij}^{\tau}$ and
$\overline{S}_{kr}^{\tau'}$, respectively.
\end{enumerate}
\end{enumerate}
\end{enumerate}

These modifications does not reduce the complexity $O(N^3)$  obtained for problem (3), 
 since the cardinal of set $P(ij;kr)$ has, in this case, the same complexity $O(N^2)$ discussed in the previous section. In effect, since for problem (4) we also have $O(N^2)$ sets $P(ij; kr)$, with $|P(ij; kr)| \le 48$ (each intersection $\overline{S}_{ij}^{\tau} \cap \overline{S}_{kr}^{\tau'}$ has at most 12 points).

\begin{remark}
Procedure for solving the restricted problem of type 2 can also be viewed as an adaptation of algorithm described in~\cite{KorMesa11} for solving the restricted problem on a tree network, which supposes that $L_p, L_q$ assume the roles of a pair of edges of the tree. If $\{L_p, L_q\}$ is a pair of linear arc segments of type 2, 
 it is possible to obtain from the network ${\cal N}$ an associated subtree $T'_{pq}$ by deleting all edges and subedges not involved in the shortest distance between points of $L_p$ and $L_q$.  If in such a subtree we consider the extreme points of $ L_p$ and $L_q$ as nodes, then the  linear arc segments  becomes edges of the tree, and the restricted problem (4) 
is an O/D 2-location problem on an edge-pair of a tree.
\end{remark}

\subsection{Complexity of the problem}

All these reasonings implies that  restricted problem (2) can be solved in $O(N^3)$ time (independently of the type of problem). Tacking into account that  
there are at most $|V| |E|$  linear arc segments in the network, the number of pairs of linear arc segments is  $O( |V|^2 |E|^2)$ (which is
 the number of  restricted problems (2) to be solved). This finally gives a complexity of   $O(|V|^2 |E|^2 N^3)$ time for solving the initial problem (1) on the overall network
(where $N =O(n^2)$ is the number of O/D pairs).

\section{Conclusions and further research}
Given a set of origin-destination pairs in the Euclidean plane, the problem of locating two transfer points on  a embedded in the plane high-speed network so that the number of covered pairs would be maximized, has been analyzed and solved in this paper. Given the lack of convexity of the mixed plane-network distances the applied approach has consisted into the decomposition of the set of feasible solutions into smaller regions in which the mixed distances are either concave or convex. This decomposition gives rise to two different restricted problems which has been analyzed and a finite dominating set for each case derived. In such a way the problem can be solved by a polynomial time algorithm which has been designed. 

Further research, still on progress, consists in applying this approach to the problem of maximizing the additional coverage regarding a set of already located transfer points on the network, i.e. the conditional location counterpart problem. An improvement in the way of approaching the model to the problem of locating new stations on a transportation network is to introduce acceleration and deceleration between transfer points. This problem also deserves further research. We note that since the computational  complexity of the corresponding problem of locating more than two transfer points  is high other methodologies as the Big Cube Small Cube algorithm~\cite{SchoScholz} could be useful for this extension.

\bigskip

\noindent{\large \bf Acknowledgement:} 
This work was partially supported by Ministerio de Educaci\'on, Ciencia e Innovaci\'on (Spain)/FEDER  under project MTM2009-14243, Ministerio de Econom\'{\i}a y Competitividad under project MTM2012-37048 and by Junta de Andaluc\'{\i}a (Spain)/FEDER under excellence projects P09-TEP-5022 and P10-FQM-5849.


\begin{thebibliography}{123}

\bibitem{Boyd}    
{\sc S.P. Boyd and L. Vanderberghe}. {\it Convex Optimization}. Cambridge University Press, 2004.

\bibitem{Haetal}
{\sc H. Hamacher,  A. Liebers,   A. Sch\"{o}bel,  D. Wagner,  F. Wagner}. {\it Locating new stops in a railway network}. Electronic Notes Theoretical Computers Science {\bf 50} (2001), 1-11

\bibitem{HoGarChen91}
{\sc J.J. Hooker, R.S. Garfinkel, C.K. Chen}. {\it Finite dominating sets for network location problems}. Operations Research {\bf 39} (1991), 100-118.

\bibitem{Kir}
{\sc F. Kirwan}. {\it Complex Algebraic Curves}. Cambridge University Press, 1992.

\bibitem{KorMesa11}
{\sc M-C. K\"{o}rner, J.A. Mesa, F. Perea, A. Sch\"{o}bel, D. Scholz}.
{\it A Maximum trip covering location problem with alternative mode of transportation on tree networks and segments}. TOP (2012):  DOI: 10.1007/s11750-012-0251-y. 

\bibitem{LaMeOr}
{\sc G. Laporte, J.A. Mesa, F.A. Ortega}. {\it Locating stations on rapid transit lines}. Computers \& Operations Research {\bf 29} (2002), 741-759.

\bibitem{LaMeOrSe}
{\sc  G. Laporte, J.A. Mesa,  F.A. Ortega,  I. Sevillano}. {\it Maximizing trip coverage in the location of a single rapid transit line alignment}. Annals of Operations Research {\bf 136} (2005),  49-61.

  
\bibitem{ReAnChurch}
{\sc  H.M. Repolho,  A.P. Antunes, R.L. Church}. {\it Optimal location of railway stations: The Lisbon-Porto High-Speed rail Line}. Transportation Science  {\bf 47} (2013), 330-343.

 \bibitem{Rockafeller}
{\sc R.T. Rockafeller}. {\it Convex Analysis}. Princeton University Press, 1970.

\bibitem{Scho}
{\sc A. Sch\"{o}bel}. {\it Locating stops along bus or railway lines. A bicriteria problem}. Annals of Operations Research {\bf 136} (2005), 211-227.

\bibitem{Schoetal}
{\sc  A. Sch\"{o}bel, H. Hamacher, A. Liebers, D. Wagner}. {\it The continuous stop location problem in public transportation networks}. Asia-Pacific Journal of Operational Research {\bf 26} (2009), 13-30.

\bibitem{SchoScholz}
 {\sc  A. Sch\"{o}bel,   D. Scholz}. {\it The Big Cube Small Cube Solution Method for Multidimensional Facility Location Problems}. Computers \& Operations Research {\bf 37} (2010),  115-122.


\bibitem{VuNe}
{\sc V.R. Vuchic, G.F. Newell}.  {\it Rapid transit interstation spacing for minimum travel time}. Transportation Science {\bf 2} (1968), 303-309.

\bibitem{Vu}
 {\sc V.R. Vuchic}.  {\it Rapid transit interstation spacings for maximum number of passengers}. Transportation Science {\bf 3} (1969), 214-232.

\end{thebibliography}
\end{document}